\documentclass[twoside,reqno]{amsart}

\usepackage{latexsym}

\newcommand{\qdn}{\hspace*{-1.5mm}}
\newcommand{\qqdn}{\hspace*{-2.5mm}}
\newcommand{\xqdn}{\hspace*{-5.0mm}}
\newcommand{\xxqdn}{\hspace*{-10mm}}

\newcommand{\sst}{\scriptstyle}

\newcommand{\ffnk}[4]{\left[\qdn\ba{#1}#3\\#4\ea{\!\Big|\:#2}\right]}

\newcommand{\binm}{\binom}

\newcommand{\nnm}{\nonumber}
\newcommand{\be}{\begin{equation}}
\newcommand{\ee}{\end{equation}}
\newcommand{\ba}{\begin{array}}
\newcommand{\ea}{\end{array}}
\newcommand{\bmn}{\begin{eqnarray}}
\newcommand{\emn}{\end{eqnarray}}
\newcommand{\bnm}{\begin{eqnarray*}}
\newcommand{\enm}{\end{eqnarray*}}
\newcommand{\bln}{\begin{subequations}}
\newcommand{\eln}{\end{subequations}}

\newtheorem{thm}{Theorem}
\newtheorem{lemm}[thm]{Lemma}
\newtheorem{corl}[thm]{Corollary}

\newtheorem{entry}{Entry}

\newcommand{\bbtm}[4]{\bibitem{kn:#1}{#2,}~{#3,}~{#4.}}
\newcommand{\cito}[1]{\cite{kn:#1}}
\newcommand{\citu}[2]{\cite[#2]{kn:#1}}

\begin{document} 
{
\title{Watson-type $_3F_2$-series
 and summation formulae involving generalized harmonic
numbers}
\author{$^{a,b}$Chuanan Wei}


\dedicatory{$^A$Department of Mathematics\\
  Shanghai Normal University, Shanghai 200234, China\\
$^B$Department of Information Technology\\
  Hainan Medical College, Haikou 571199, China}

\thanks{\emph{Email address}: weichuanan78@ 163.com}


\begin{abstract}
In terms of the derivative operator and Watson-type $_3F_2$-series
identities, three families of summation formulae involving
generalized harmonic numbers are established.
\end{abstract}

\maketitle\thispagestyle{empty}
\markboth{C. Wei}
         {Summation formulae involving generalized harmonic numbers}

\section{Introduction}
For a complex variable $x$, define the shifted factorial to be
\[(x)_{0}=0\quad \text{and}\quad (x)_{n}
=x(x+1)\cdots(x+n-1)\quad \text{with}\quad n\in\mathbb{N}.\]
 Following Andrews, Askey and Roy~\citu{andrews-r}{Chapter 2}, define the hypergeometric series by
\[\qdn_{1+r}F_s\ffnk{cccc}{z}{a_0,&a_1,&\cdots,&a_r}{&b_1,&\cdots,&b_s}
 \:=\:\sum_{k=0}^\infty\frac{(a_0)_k(a_1)_k\cdots(a_r)_k}{k!(b_1)_k\cdots(b_s)_k}z^k,\]
where $\{a_{i}\}_{i\geq0}$ and $\{b_{j}\}_{j\geq1}$ are complex
parameters such that no zero factors appear in the denominators of
the summand on the right hand side. Then Watson's $_3F_2$-series
identity (cf. \citu{andrews-r}{p. 148}) can be stated as
 \bmn\label{watson}
 _3F_2\ffnk{cccc}{1}{a,b,c}{\frac{1+a+b}{2},2c}
=\frac{\Gamma(\frac{1}{2})\Gamma(\frac{1+a+b}{2})\Gamma(\frac{1}{2}+c)\Gamma(\frac{1-a-b}{2}+c)}
{\Gamma(\frac{1+a}{2})\Gamma(\frac{1+b}{2})\Gamma(\frac{1-a}{2}+c)\Gamma(\frac{1-b}{2}+c)},
 \emn
where $Re(1-a-b+2c)>0$ and $\Gamma(x)$ is the well-known gamma
function
\[\xqdn\Gamma(x)=\int_{0}^{\infty}t^{x-1}e^{-t}dt\:\:\text{with}\:\:Re(x)>0.\]
A terminating form of \eqref{watson} due to Bailey \cito{bailey} can
be expressed as
 \bmn\label{watson-ter}
 _3F_2\ffnk{cccc}{1}{a,b,-n}{\frac{1+a+b}{2},-2n}
=\frac{(\frac{1+a}{2})_n(\frac{1+b}{2})_n}{(\frac{1}{2})_n(\frac{1+a+b}{2})_n}.
 \emn

For a complex number $x$ and a positive integer $\ell$, define
generalized harmonic numbers of $\ell$-order to be
\[H_{0}^{\langle \ell\rangle}(x)=0
\quad\text{and}\quad
 H_{n}^{\langle\ell\rangle}(x)=\sum_{k=1}^n\frac{1}{(x+k)^{\ell}}
 \quad\text{with}\quad n\in\mathbb{N}.\]
When $x=0$, they become harmonic numbers of $\ell$-order
\[H_{0}^{\langle \ell\rangle}=0
\quad\text{and}\quad
  H_{n}^{\langle \ell\rangle}
  =\sum_{k=1}^n\frac{1}{k^{\ell}} \quad\text{with}\quad n\in\mathbb{N}.\]
  Fixing $\ell=1$ in $H_{0}^{\langle \ell\rangle}(x)$ and $H_{n}^{\langle \ell\rangle}(x)$, we obtain
generalized harmonic numbers
\[H_{0}(x)=0
\quad\text{and}\quad H_{n}(x)
  =\sum_{k=1}^n\frac{1}{x+k} \quad\text{with}\quad n\in\mathbb{N}.\]
When $x=0$, they reduce to classical harmonic numbers
\[H_{0}=0\quad\text{and}\quad
H_{n}=\sum_{k=1}^n\frac{1}{k} \quad\text{with}\quad
n\in\mathbb{N}.\]

 For a differentiable function $f(x)$, define the derivative operator
$\mathcal{D}_x$ by
 \bnm
\mathcal{D}_xf(x)=\frac{d}{dx}f(x).
 \enm
In order to explain the relation of the derivative operator and
generalized harmonic numbers, we introduce the following lemma.

\begin{lemm} \label{lemm-a}
 Let $x$ and  $\{a_j,b_j,c_j,d_j\}_{j=1}^s$ be all complex
numbers. Then
 \bnm
\mathcal{D}_x\prod_{j=1}^s\frac{a_jx+b_j}{c_jx+d_j}=\prod_{j=1}^s\frac{a_jx+b_j}{c_jx+d_j}
 \sum_{j=1}^s\frac{a_jd_j-b_jc_j}{(a_jx+b_j)(c_jx+d_j)}.
 \enm
\end{lemm}

\begin{proof}
It is not difficult to verify the case $s=1$ of Lemma \ref{lemm-a}.
Suppose that
 \bnm
\mathcal{D}_x\prod_{j=1}^m\frac{a_jx+b_j}{c_jx+d_j}=\prod_{j=1}^m\frac{a_jx+b_j}{c_jx+d_j}
 \sum_{j=1}^m\frac{a_jd_j-b_jc_j}{(a_jx+b_j)(c_jx+d_j)}
 \enm
is true. We can proceed as follows:
 \bnm
&&\xqdn\mathcal{D}_x\prod_{j=1}^{m+1}\frac{a_jx+b_j}{c_jx+d_j}=
 \mathcal{D}_x\bigg\{\prod_{j=1}^{m}\frac{a_jx+b_j}{c_jx+d_j}\frac{a_{m+1}x+b_{m+1}}{c_{m+1}x+d_{m+1}}\bigg\}\\
&&\xqdn\:=\:\frac{a_{m+1}x+b_{m+1}}{c_{m+1}x+d_{m+1}}\mathcal{D}_x\prod_{j=1}^{m}\frac{a_jx+b_j}{c_jx+d_j}
+\prod_{j=1}^{m}\frac{a_jx+b_j}{c_jx+d_j}\mathcal{D}_x\frac{a_{m+1}x+b_{m+1}}{c_{m+1}x+d_{m+1}}\\
&&\xqdn\:=\:\frac{a_{m+1}x+b_{m+1}}{c_{m+1}x+d_{m+1}}\prod_{j=1}^m\frac{a_jx+b_j}{c_jx+d_j}\sum_{j=1}^m\frac{a_jd_j-b_jc_j}{(a_jx+b_j)(c_jx+d_j)}\\
&&\xqdn\:+\:\prod_{j=1}^{m}\frac{a_jx+b_j}{c_jx+d_j}\frac{a_{m+1}d_{m+1}-b_{m+1}c_{m+1}}{(c_{m+1}x+d_{m+1})^2}\\
&&\xqdn\:=\:\prod_{j=1}^{m+1}\frac{a_jx+b_j}{c_jx+d_j}
 \bigg\{\sum_{j=1}^m\frac{a_jd_j-b_jc_j}{(a_jx+b_j)(c_jx+d_j)}+\frac{a_{m+1}d_{m+1}-b_{m+1}c_{m+1}}{(a_{m+1}x+b_{m+1})(c_{m+1}x+d_{m+1})}\bigg\}\\
&&\xqdn\:=\:\prod_{j=1}^{m+1}\frac{a_jx+b_j}{c_jx+d_j}
 \sum_{j=1}^{m+1}\frac{a_jd_j-b_jc_j}{(a_jx+b_j)(c_jx+d_j)}.
 \enm
This proves Lemma \ref{lemm-a} inductively.
\end{proof}

Setting $a_j=1,b_j=r-j+1,c_j=0,d_j=j$ in Lemma \ref{lemm-a}, it is
easy to find that
$$\mathcal{D}_x\:\binm{x+r}{s}=\binm{x+r}{s}\big\{H_r(x)-H_{r-s}(x)\big\},$$
where $r,s\in\mathbb{N}_0$ with $s\leq r$. Besides, we have the
following relation:
$$\mathcal{D}_xH_{n}^{\langle \ell\rangle}(x)=-\ell H_{n}^{\langle
\ell+1\rangle}(x).$$

As pointed out by Richard Askey (cf. \cito{andrews}), expressing
harmonic numbers in accordance with differentiation of binomial
coefficients can be traced back to Issac Newton.
 In 2003, Paule and Schneider
\cito{paule} computed the family of series:
 \bnm
\quad
W_n(\alpha)=\sum_{k=0}^n\binm{n}{k}^{\alpha}\{1+\alpha(n-2k)H_k\}
 \enm
with $\alpha=1,2,3,4,5$ by combining this way with Zeilberger's
algorithm for definite hypergeometric sums. According to the
derivative operator and the hypergeometric form of Andrews'
$q$-series transformation, Krattenthaler and Rivoal
\cito{krattenthaler} deduced general Paule-Schneider type identities
with $\alpha$ being a positive integer.  More results from
differentiation of binomial coefficients can be seen in the papers
\cite{kn:sofo-a,kn:wang-b,kn:wei-a,kn:wei-c,kn:wei-d}. For different
ways and related harmonic number identities, the reader may refer to
\cite{kn:chen,kn:kronenburg-a,
kn:kronenburg-b,kn:schneider,kn:sofo-b,kn:wang-a}. It should be
mentioned that Sun \cito{sun} showed recently some congruence
relations concerning harmonic numbers to us.

Inspired by the work just mentioned, we shall explore, by means of
the derivative operator and Watson-type $_3F_2$-series identities,
closed expressions for the following three families of series:
 \bnm
 &&\xqdn\sum_{k=0}^{n}(-1)^k\binm{n}{k}\frac{\binm{2x+n+k}{k}}{\binm{x+k}{k}}k^tH_k,\\
 &&\xqdn\sum_{k=0}^{n}\binm{2n-k}{n}\binm{x+k}{k}k^tH_{k}^{\langle2\rangle}(x),\\
 &&\xqdn\sum_{k=0}^{n}\binm{2n-k}{n}\binm{x+k}{k}k^tH_{k}^2(x),\\
 \enm
where $t\in\mathbb{N}_0$. For limit of space, our explicit formulae
are offered only for $t=0,1,2$. In order to make the reader have a
taste, we enumerate, above all, three concise harmonic number
identities from the special cases of them as follows:
  \bnm
&&\sum_{k=0}^n(-2)^k\binm{n}{k}H_{k}=
 \begin{cases}
 2H_{2m}-H_m,&\qdn n=2m;
\\[2mm]
H_m-2H_{2m+1},&\qdn n=2m+1,
\end{cases}\\
&&\:\sum_{k=0}^{n}\binm{2n-k}{n}H_{k}^{\langle2\rangle}=\frac{\binm{2n+1}{n}}{2}
\big\{H_{2n+1}^{\langle2\rangle}-(H_{2n+1}-H_n)^2\big\},\\
&&\:\sum_{k=0}^{n}\binm{2n-k}{n}H^2_{k}=\frac{\binm{2n+1}{n}}{2}
\big\{2H_{n+1}^{\langle2\rangle}-H_{2n+1}^{\langle2\rangle}+2H^2_{n+1}-H^2_n
\\&&\qquad\qquad\qquad\qquad\qquad\qquad
+H_{2n+1}(H_{2n+1}-4H_{n+1}+2H_n)\big\}.
 \enm

\section{The first family of summation formulae involving\\ generalized harmonic numbers}
\begin{lemm}\label{lemm-b}
Let $a$, $b$ and $c$ be all complex numbers. Then
 \bnm
 _3F_2\ffnk{cccc}{1}{a,b,c}{\frac{1+a+b}{2},2c-1}
&&\xqdn\!=\frac{\Gamma(\frac{1}{2})\Gamma(\frac{1+a+b}{2})\Gamma(c-\frac{1}{2})\Gamma(c-\frac{1+a+b}{2})}
{\Gamma(\frac{a}{2})\Gamma(\frac{b}{2})\Gamma(c-\frac{a}{2})\Gamma(c-\frac{b}{2})}\\
&&\xqdn\!+\:\frac{\Gamma(\frac{1}{2})\Gamma(\frac{1+a+b}{2})\Gamma(c-\frac{1}{2})\Gamma(c-\frac{1+a+b}{2})}
{\Gamma(\frac{1+a}{2})\Gamma(\frac{1+b}{2})\Gamma(c-\frac{1+a}{2})\Gamma(c-\frac{1+b}{2})},
 \enm
 where the convergence condition is $Re(2c-a-b-1)>0$.
\end{lemm}

\begin{proof}
Recall Whipple's $_3F_2$-series identity (cf. \citu{andrews-r}{p.
149}):
 \bnm
 \:\:_3F_2\ffnk{cccc}{1}{a,1-a,b}{c,1+2b-c}
=\frac{\pi 2^{1-2b}\Gamma(c)\Gamma(1+2b-c)}
{\Gamma(\frac{a+c}{2})\Gamma(\frac{1+a-c}{2}+b)\Gamma(\frac{1-a+c}{2})\Gamma(\frac{2-a-c}{2}+b)}
 \enm
provided that $Re(b)>0$. Perform the replacement $a\to 1+a$ in the
last equation to get
 \bnm
\quad _3F_2\ffnk{cccc}{1}{1+a,-a,b}{c,1+2b-c} =\frac{\pi
2^{1-2b}\Gamma(c)\Gamma(1+2b-c)}
{\Gamma(\frac{1+a+c}{2})\Gamma(\frac{2+a-c}{2}+b)\Gamma(\frac{c-a}{2})\Gamma(\frac{1-a-c}{2}+b)}.
 \enm
 The linear combination of the last two equations
gives
 \bmn\label{whipple-a}
 _3F_2\ffnk{cccc}{1}{a,-a,b}{c,1+2b-c}&&\xqdn\!=\frac{\pi\Gamma(c)\Gamma(1+2b-c)}
{4^b\Gamma(\frac{a+c}{2})\Gamma(\frac{1+a-c}{2}+b)\Gamma(\frac{1-a+c}{2})\Gamma(\frac{2-a-c}{2}+b)}
\nnm\\&&\xqdn\!+\:\frac{\pi\Gamma(c)\Gamma(1+2b-c)}
{4^b\Gamma(\frac{1+a+c}{2})\Gamma(\frac{2+a-c}{2}+b)\Gamma(\frac{c-a}{2})\Gamma(\frac{1-a-c}{2}+b)}.
 \emn
 In terms of Kummer's transformation formula (cf.
\citu{andrews-r}{p. 143}):
  \bmn \label{kummer}\quad
\xxqdn\qdn_3F_2\ffnk{cccc}{1}{a,b,c}{d,e}
 &&\xqdn\!\!=\frac{\Gamma(d)\Gamma(e)\Gamma(d+e-a-b-c)}{\Gamma(a)\Gamma(d+e-a-b)\Gamma(d+e-a-c)}
 \nnm\\
&&\xqdn\!\!\times\:\,{_3F_2}\ffnk{cccc}{1}{d-a,e-a,d+e-a-b-c}{d+e-a-b,d+e-a-c},
 \emn
we gain
  \bnm
  3F_2\ffnk{cccc}{1}{c,b,a}{\frac{1+a+b}{2},2c-1}
 &&\xqdn\!=\frac{\Gamma(\frac{1+a+b}{2})\Gamma(2c-1)\Gamma(c-\frac{1+a+b}{2})}{\Gamma(c)\Gamma(c-\frac{1+b-a}{2})\Gamma(c-\frac{1+a-b}{2})}\\
 &&\xqdn\!\times\:\,{_3F_2}\ffnk{cccc}{1}{\frac{1+a+b}{2}-c,c-1,c-\frac{1+a+b}{2}}{c-\frac{1+b-a}{2},c-\frac{1+a-b}{2}}.
 \enm
Calculating the series on the right hand side by \eqref{whipple-a},
we achieve Lemma \ref{lemm-b} to complete the proof.
\end{proof}

\begin{thm} \label{thm-a}
Let $x$ be a complex number. Then
 \bnm
\sum_{k=0}^n(-1)^k\binm{n}{k}\frac{\binm{2x+n+k}{k}}{\binm{x+k}{k}}H_{k}
=\begin{cases} H_m(x)+2H_{2m}-H_m,&\qdn n=2m;
\\[2mm]
H_m-2H_{2m+1}-H_m(x),&\qdn n=2m+1.
\end{cases}
 \enm
\end{thm}

\begin{proof}
The case $a=-n$, $b=2x+n+1$ and $c=y+1$ of Lemma \ref{lemm-b} reads
as
 \bnm
\sum_{k=0}^n(-1)^k\binm{n}{k}\frac{\binm{2x+n+k}{k}\binm{y+k}{k}}{\binm{x+k}{k}\binm{2y+k}{k}}
=\begin{cases}
\frac{\binm{x-y+m}{m}\binm{m-\frac{1}{2}}{m}}{\binm{x+m}{m}\binm{y+m-\frac{1}{2}}{m}},&\qdn
n=2m;
\\[3mm]
\frac{-1}{2y+1}\frac{\binm{x-y+m}{m}\binm{m+\frac{1}{2}}{m}}{\binm{x+m}{m}\binm{y+m+\frac{1}{2}}{m}},&\qdn
n=2m+1.
\end{cases}
 \enm
 Applying the derivative operator $\mathcal{D}_y$ to both sides of
 it, we attain
 \bnm
&&\xqdn\sum_{k=0}^n(-1)^k\binm{n}{k}\frac{\binm{2x+n+k}{k}\binm{y+k}{k}}{\binm{x+k}{k}\binm{2y+k}{k}}
\big\{2H_k(2y)-H_k(y)\big\}\\
&&\xqdn\:=\:\begin{cases}
\frac{\binm{x-y+m}{m}\binm{m-\frac{1}{2}}{m}}{\binm{x+m}{m}\binm{y+m-\frac{1}{2}}{m}}
\big\{H_m(x-y)+H_m(y-\frac{1}{2})\big\},&\qdn n=2m;
\\[3mm]
\frac{-1}{2y+1}\frac{\binm{x-y+m}{m}\binm{m+\frac{1}{2}}{m}}{\binm{x+m}{m}\binm{y+m+\frac{1}{2}}{m}}
\big\{H_m(x-y)+H_{m+1}(y-\frac{1}{2})\big\},&\qdn n=2m+1.
\end{cases}
 \enm
Choosing $y=0$ in the last equation, we obtain Theorem \ref{thm-a}
to finish the proof.
\end{proof}

Taking $x=p$ with $p\in \mathbb{N}_0$ in Theorem \ref{thm-a}, we
have the summation formula on harmonic numbers.

\begin{corl} \label{corl-a}
Let $p$ be a nonnegative integer. Then
 \bnm
&&\xqdn\sum_{k=0}^n(-1)^k\binm{n}{k}\binm{2p+n+k}{p+k}H_{k}\\
&&\,\,\xqdn=\binm{2p+n}{p}\begin{cases}
2H_{2m}-H_m+H_{p+m}-H_p,&\qdn n=2m;
\\[2mm]
H_p-H_{p+m}+H_m-2H_{2m+1},&\qdn n=2m+1.
\end{cases}
 \enm
\end{corl}

 When $x\to\infty$, Theorem \ref{thm-a} reduces to the concise harmonic
 number identity:
 \bnm
  \sum_{k=0}^n(-2)^k\binm{n}{k}H_{k}=
 \begin{cases}
 2H_{2m}-H_m,&\qdn n=2m;
\\[2mm]
H_m-2H_{2m+1},&\qdn n=2m+1.
\end{cases}
 \enm

\begin{lemm}\label{lemm-c}
Let $a$, $b$ and $c$ be all complex numbers. Then
 \bnm
 \sum_{k=0}^{\infty}k\frac{(a)_k(b)_k(c)_k}{k!(\frac{1+a+b}{2})_k(2c-1)_k}
&&\xqdn\!=ab\frac{\Gamma(\frac{1}{2})\Gamma(\frac{1+a+b}{2})\Gamma(c-\frac{1}{2})\Gamma(c-\frac{3+a+b}{2})}
{\Gamma(\frac{1+a}{2})\Gamma(\frac{1+b}{2})\Gamma(c-\frac{1+a}{2})\Gamma(c-\frac{1+b}{2})}\\
&&\xqdn\!+\:\{2c^2-(a+b+3)c+(a+1)(b+1)\}\\
&&\xqdn\!\times\:\frac{\Gamma(\frac{1}{2})\Gamma(\frac{1+a+b}{2})\Gamma(c-\frac{1}{2})\Gamma(c-\frac{3+a+b}{2})}
{\Gamma(\frac{a}{2})\Gamma(\frac{b}{2})\Gamma(c-\frac{a}{2})\Gamma(c-\frac{b}{2})},
 \enm
where the convergence condition is $Re(2c-a-b-3)>0$.
\end{lemm}

\begin{proof}
Employ the substitution $a\to 1+a$ in \eqref{whipple-a} to get
 \bnm
 _3F_2\ffnk{cccc}{1}{1+a,-1-a,b}{c,1+2b-c}&&\xqdn\!=\frac{\pi\Gamma(c)\Gamma(1+2b-c)}
{4^b\Gamma(\frac{1+a+c}{2})\Gamma(\frac{2+a-c}{2}+b)\Gamma(\frac{c-a}{2})\Gamma(b-\frac{a+c-1}{2})}
\\&&\xqdn\!+\:\frac{\pi\Gamma(c)\Gamma(1+2b-c)}
{4^b\Gamma(\frac{2+a+c}{2})\Gamma(\frac{3+a-c}{2}+b)\Gamma(\frac{c-a-1}{2})\Gamma(b-\frac{a+c}{2})}.
 \enm
The linear combination of \eqref{whipple-a} and the last equation
provides
 \bmn\label{whipple-b}
 \xqdn_3F_2\ffnk{cccc}{1}{a,-1-a,b}{c,1+2b-c}&&\xqdn\!=\frac{\pi\Gamma(c)\Gamma(1+2b-c)}
 {4^b\Gamma(\frac{1+a+c}{2})\Gamma(\frac{2+a-c}{2}+b)\Gamma(\frac{c-a}{2})\Gamma(\frac{1-a-c}{2}+b)}
 \nnm\\\nnm
 &&\xqdn\!+\:\frac{a^2+a+(1+2b-c)c}{(1+a+2b-c)(a+c)}
\\&&\xqdn\!\times\:\frac{\pi\Gamma(c)\Gamma(1+2b-c)}
{4^b\Gamma(\frac{a+c}{2})\Gamma(\frac{1+a-c}{2}+b)\Gamma(\frac{1-a+c}{2})\Gamma(\frac{2-a-c}{2}+b)}.
 \emn
In accordance with \eqref{kummer}, we gain
  \bnm
  \qqdn\xqdn_3F_2\ffnk{cccc}{1}{c,b,a}{\frac{1+a+b}{2},2c-2}
 &&\xqdn\!=\frac{\Gamma(\frac{1+a+b}{2})\Gamma(2c-2)\Gamma(c-\frac{3+a+b}{2})}{\Gamma(c)\Gamma(c-\frac{3+b-a}{2})\Gamma(c-\frac{3+a-b}{2})}\\
 &&\xqdn\!\times\:{_3F_2}\ffnk{cccc}{1}{\frac{1+a+b}{2}-c,c-2,c-\frac{3+a+b}{2}}{c-\frac{3+b-a}{2},c-\frac{3+a-b}{2}}.
 \enm
Computing the series on the right hand side by \eqref{whipple-b}, we
achieve
 \bmn\label{watson-a}
 _3F_2\ffnk{cccc}{1}{a,b,c}{\frac{1+a+b}{2},2c-2}
&&\xqdn\!=\frac{2}{c-1}\frac{\Gamma(\frac{1}{2})\Gamma(\frac{1+a+b}{2})\Gamma(c-\frac{1}{2})\Gamma(c-\frac{3+a+b}{2})}
{\Gamma(\frac{a}{2})\Gamma(\frac{b}{2})\Gamma(c-\frac{2+a}{2})\Gamma(c-\frac{2+b}{2})}
\nnm\\\nnm &&\xqdn\!+\:\frac{2c^2-(a+b+5)c+ab+a+b+3}{2(c-1)}
 \\&&\xqdn\!\times\:
\frac{\Gamma(\frac{1}{2})\Gamma(\frac{1+a+b}{2})\Gamma(c-\frac{1}{2})\Gamma(c-\frac{3+a+b}{2})}
{\Gamma(\frac{1+a}{2})\Gamma(\frac{1+b}{2})\Gamma(c-\frac{1+a}{2})\Gamma(c-\frac{1+b}{2})}.
 \emn
It is routine to show that
 \bnm
 \sum_{k=0}^{\infty}k\frac{(a)_k(b)_k(c)_k}{k!(\frac{1+a+b}{2})_k(2c-1)_k}
 &&\xqdn\!=\sum_{k=1}^{\infty}\frac{(a)_k(b)_k(c)_k}{(k-1)!(\frac{1+a+b}{2})_k(2c-1)_k}\\
 &&\xqdn\!=\sum_{k=0}^{\infty}\frac{(a)_{k+1}(b)_{k+1}(c)_{k+1}}{k!(\frac{1+a+b}{2})_{k+1}(2c-1)_{k+1}}\\
&&\xqdn\!=\frac{2abc}{(1+a+b)(2c-1)}{_3F_2}\ffnk{cccc}{1}{1+a,1+b,1+c}{\frac{3+a+b}{2},2c}.
 \enm
 Evaluating the
series on the right hand side by \eqref{watson-a}, we attain Lemma
\ref{lemm-c} to complete the proof.
\end{proof}

\begin{thm} \label{thm-b}
Let $x$ be a complex number. Then
 \bnm
&&\qqdn\sum_{k=0}^n(-1)^k\binm{n}{k}\frac{\binm{2x+n+k}{k}}{\binm{x+k}{k}}kH_{k}=\frac{n(2x+n+1)}{x+1}\\
&&\qqdn\,\,\times\begin{cases} H_{m-1}(x+1)+2H_{2m}-H_m,&\qdn n=2m;
\\[2mm]
H_m-2H_{2m}-H_m(x+1)-\frac{x+1}{(2m+1)(x+m+1)},&\qdn n=2m+1.
\end{cases}
 \enm
\end{thm}

\begin{proof}
The case $a=-n$, $b=2x+n+1$ and $c=y+1$ of Lemma \ref{lemm-c} can be
written as
 \bnm
&&\xxqdn\sum_{k=0}^n(-1)^kk\binm{n}{k}\frac{\binm{2x+n+k}{k}\binm{y+k}{k}}{\binm{x+k}{k}\binm{2y+k}{k}}\\
&&\xxqdn\,\,=\:\begin{cases}
 \frac{2m(2x+2m+1)}{x-y+1}
\frac{\binm{x-y+m}{m}\binm{m-\frac{1}{2}}{m}}{\binm{x+m}{m}\binm{y+m-\frac{1}{2}}{m}},&\qdn
n=2m;
\\[3mm]
\frac{2\{y^2-xy-(2m+1)(x+m+1)\}}{(x-y+1)(2y+1)}
\frac{\binm{x-y+m}{m}\binm{m+\frac{1}{2}}{m}}{\binm{x+m}{m}\binm{y+m+\frac{1}{2}}{m}},&\qdn
n=2m+1.
\end{cases}
 \enm
 Applying the derivative operator $\mathcal{D}_y$ to both sides of
 it, we obtain
 \bnm
&&\xqdn
\sum_{k=0}^n(-1)^k\binm{n}{k}\frac{\binm{2x+n+k}{k}\binm{y+k}{k}}{\binm{x+k}{k}\binm{2y+k}{k}}
k\big\{2H_k(2y)-H_k(y)\big\}\\
 &&\xqdn\:=\:\begin{cases}
\frac{2m(2x+2m+1)}{x-y+1}\frac{\binm{x-y+m}{m}\binm{m-\frac{1}{2}}{m}}{\binm{x+m}{m}\binm{y+m-\frac{1}{2}}{m}}
\big\{H_{m-1}(x-y+1)+H_m(y-\frac{1}{2})\big\},&\qdn n=2m;
\\[3mm]
\frac{2\alpha_m(x,y)}{(x-y+1)(2y+1)}
\frac{\binm{x-y+m}{m}\binm{m+\frac{1}{2}}{m}}{\binm{x+m}{m}\binm{y+m+\frac{1}{2}}{m}}\big\{H_m(x-y)+H_m(y+\frac{1}{2})-\beta_m(x,y)\big\}
,&\qdn n=2m+1,
\end{cases}
 \enm
where the corresponding expressions are
 \bnm
&&\alpha_m(x,y)=y^2-xy-(2m+1)(x+m+1),\\
&&\beta_m(x,y)=\frac{m(2x+2m+3)(2x-4y+1)+(x-y+1)^2}{(x-y+1)(2y+1)\{y^2-xy-(2m+1)(x+m+1)\}}.
 \enm
 Selecting $y=0$ in the last
equation, we get Theorem \ref{thm-b} to finish the proof.
\end{proof}

Fixing $x=p$ with $p\in \mathbb{N}_0$ in Theorem \ref{thm-b}, we
deduce the summation formula on harmonic numbers.

\begin{corl} \label{corl-b}
Let $p$ be a nonnegative integer. Then
 \bnm
&&\xxqdn\sum_{k=0}^n(-1)^k\binm{n}{k}\binm{2p+n+k}{p+k}kH_{k}=n\binm{2p+n+1}{p+1}\\
&&\xxqdn\,\,\times\begin{cases} 2H_{2m}-H_m+H_{p+m}-H_{p+1},&\qdn
n=2m;
\\[2mm]
H_m-2H_{2m}+H_{p+1}-H_{p+m+1}-\frac{p+1}{(2m+1)(p+m+1)},&\qdn
n=2m+1.
\end{cases}
 \enm
\end{corl}

When $x\to\infty$, Theorem \ref{thm-b} produces the concise harmonic
 number identity:
\bnm
 \:\quad \sum_{k=0}^n(-2)^k\binm{n}{k}kH_{k}=
 \begin{cases}
 4m(2H_{2m}-H_m),&\qdn n=2m;
\\[2mm]
2(2m+1)(H_m-2H_{2m+1})+2,&\qdn n=2m+1.
\end{cases}
 \enm

\begin{lemm}\label{lemm-d}
Let $a$, $b$ and $c$ be all complex numbers. Then
 \bnm
 \sum_{k=0}^{\infty}k^2\frac{(a)_k(b)_k(c)_k}{k!(\frac{1+a+b}{2})_k(2c-1)_k}
&&\xqdn\!=\Phi(a,b,c)\frac{\Gamma(\frac{3}{2})\Gamma(\frac{1+a+b}{2})\Gamma(c-\frac{1}{2})\Gamma(c-\frac{5+a+b}{2})}
{\Gamma(\frac{1+a}{2})\Gamma(\frac{1+b}{2})\Gamma(c-\frac{1+a}{2})\Gamma(c-\frac{1+b}{2})}\\
&&\xqdn\!+\:\Psi(a,b,c)\frac{\Gamma(\frac{3}{2})\Gamma(\frac{1+a+b}{2})\Gamma(c-\frac{1}{2})\Gamma(c-\frac{5+a+b}{2})}
{\Gamma(\frac{a}{2})\Gamma(\frac{b}{2})\Gamma(c-\frac{a}{2})\Gamma(c-\frac{b}{2})},
 \enm
 where the symbols on the right hand side stand for
\bnm
 &&\xxqdn\Phi(a,b,c)=ab\{2c^2-(1+a+b)c+2(a+b+ab-1)\},\\
  &&\xxqdn\Psi(a,b,c)=4c^3+2(a+b+3ab-5)c^2\\
  &&\quad\:-\,(3a^2b+3ab^2+11ab+2a^2+2a+2b^2+2b-8)c\\
  &&\quad\:+\,2(1+a)(1+b)(a+b+ab-1)
 \enm
 and the parameters satisfy the condition $Re(2c-a-b-5)>0$.
\end{lemm}

\begin{proof}
Replace $a$ by $1+a$ in \eqref{whipple-b} to gain
 \bnm
\qdn_3F_2\ffnk{cccc}{1}{1+a,-2-a,b}{c,1+2b-c}&&\xqdn\!=\frac{\pi\Gamma(c)\Gamma(1+2b-c)}
{4^b\Gamma(\frac{2+a+c}{2})\Gamma(\frac{3+a-c}{2}+b)\Gamma(\frac{c-a-1}{2})\Gamma(b-\frac{a+c}{2})}
\nnm\\\nnm
 &&\xqdn\!+\:\frac{a^2+3a+2+(1+2b-c)c}{(2+a+2b-c)(1+a+c)}\\
 &&\xqdn\!\times\:
\frac{\pi\Gamma(c)\Gamma(1+2b-c)}
{4^b\Gamma(\frac{1+a+c}{2})\Gamma(\frac{2+a-c}{2}+b)\Gamma(\frac{c-a}{2})\Gamma(\frac{1-a-c}{2}+b)}.
 \enm
The linear combination of \eqref{whipple-b} and the last equation
offers
 \bmn\label{whipple-c}
&&\xxqdn_3F_2\ffnk{cccc}{1}{a,-2-a,b}{c,1+2b-c}
 \nnm\\\nnm
&&\xxqdn\:=\:\frac{(2+a)(1+a+b)+(1+2b-c)c}{(2+a+2b-c)(1+a+c)}
\frac{\pi\Gamma(c)\Gamma(1+2b-c)}
{4^b\Gamma(\frac{1+a+c}{2})\Gamma(\frac{2+a-c}{2}+b)\Gamma(\frac{c-a}{2})\Gamma(\frac{1-a-c}{2}+b)}
\\&&\xxqdn\:+\:\:\frac{a(1+a-b)+(1+2b-c)c}{(a-2b+c)(1+a-c)}\frac{\pi\Gamma(c)\Gamma(1+2b-c)}
{4^b\Gamma(\frac{2+a+c}{2})\Gamma(\frac{3+a-c}{2}+b)\Gamma(\frac{c-a-1}{2})\Gamma(b-\frac{a+c}{2})}.
 \emn
 According to \eqref{kummer}, we achieve
  \bnm
  _3F_2\ffnk{cccc}{1}{c,b,a}{\frac{1+a+b}{2},2c-3}
 &&\xqdn\!=\frac{\Gamma(\frac{1+a+b}{2})\Gamma(2c-3)\Gamma(c-\frac{5+a+b}{2})}{\Gamma(c)\Gamma(c-\frac{5+b-a}{2})\Gamma(c-\frac{5+a-b}{2})}\\
 &&\xqdn\!\times\:{_3F_2}\ffnk{cccc}{1}{\frac{1+a+b}{2}-c,c-3,c-\frac{5+a+b}{2}}{c-\frac{5+b-a}{2},c-\frac{5+a-b}{2}}.
 \enm
Reckoning the series on the right hand side by \eqref{whipple-c}, we
attain
 \bmn\label{watson-b}
 &&\xxqdn _3F_2\ffnk{cccc}{1}{a,b,c}{\frac{1+a+b}{2},2c-3}
  \nnm\\\nnm
&&\xxqdn\:=\:\Big\{1+\tfrac{3ab(c-3)}{(c-1)(2c-a-5)(2c-b-5)}\Big\}
\frac{\Gamma(\frac{1}{2})\Gamma(\frac{1+a+b}{2})\Gamma(c-\frac{5}{2})\Gamma(c-\frac{5+a+b}{2})}
{\Gamma(\frac{1+a}{2})\Gamma(\frac{1+b}{2})\Gamma(c-\frac{5+a}{2})\Gamma(c-\frac{5+b}{2})}
\\&&\xxqdn\:+\:\:
\tfrac{2\{6c^2-3c(a+b+7)+2ab+5a+5b+17\}}{(c-1)(2c-a-4)(2c-b-4)}
\frac{\Gamma(\frac{1}{2})\Gamma(\frac{1+a+b}{2})\Gamma(c-\frac{3}{2})\Gamma(c-\frac{5+a+b}{2})}
{\Gamma(\frac{a}{2})\Gamma(\frac{b}{2})\Gamma(c-\frac{4+a}{2})\Gamma(c-\frac{4+b}{2})}.
 \emn
It is easy to see that
 \bnm
 \sum_{k=0}^{\infty}k^2\frac{(a)_k(b)_k(c)_k}{k!(\frac{1+a+b}{2})_k(2c-1)_k}
 &&\xqdn\!=\sum_{k=1}^{\infty}k\frac{(a)_k(b)_k(c)_k}{(k-1)!(\frac{1+a+b}{2})_k(2c-1)_k}\\
 &&\xqdn\!=\sum_{k=0}^{\infty}(k+1)\frac{(a)_{k+1}(b)_{k+1}(c)_{k+1}}{k!(\frac{1+a+b}{2})_{k+1}(2c-1)_{k+1}}\\
&&\xqdn\!=\frac{2abc}{(1+a+b)(2c-1)}{_3F_2}\ffnk{cccc}{1}{1+a,1+b,1+c}{\frac{3+a+b}{2},2c}\\
 &&\xqdn\!+\frac{2abc}{(1+a+b)(2c-1)}
  \sum_{k=1}^{\infty}\frac{(1+a)_{k}(1+b)_{k}(1+c)_{k}}{(k-1)!(\frac{3+a+b}{2})_{k}(2c)_{k}}\\
&&\xqdn\!=\frac{2abc}{(1+a+b)(2c-1)}{_3F_2}\ffnk{cccc}{1}{1+a,1+b,1+c}{\frac{3+a+b}{2},2c}\\
&&\xqdn\!+\frac{2ab(1+a)(1+b)(1+c)}{(1+a+b)(3+a+b)(2c-1)}{_3F_2}\ffnk{cccc}{1}{2+a,2+b,2+c}{\frac{5+a+b}{2},1+2c}.
 \enm
 Calculating the two
series on the right hand side by \eqref{watson-a} and
\eqref{watson-b}, we obtain Lemma \ref{lemm-d} to complete the
proof.
\end{proof}

\begin{thm} \label{thm-c}
Let $x$ be a complex number. Then
 \bnm
&&\xxqdn\sum_{k=0}^n(-1)^k\binm{n}{k}\frac{\binm{2x+n+k}{k}}{\binm{x+k}{k}}k^2H_{k}=\frac{n(2x+n+1)(n^2+n+2nx-x)}{(x+1)(x+2)}\\
&&\xxqdn\,\,\times\begin{cases}
2H_{2m}-H_m+H_{m-2}(x+2)+\frac{1-x}{2m(2x+2m+1)-x},&\qdn n=2m;
\\[2mm]
H_m-2H_{2m}-H_{m-1}(x+2)-\frac{x+3}{2m(2x+2m+3)+x+2},&\qdn n=2m+1.
\end{cases}
 \enm
\end{thm}

\begin{proof}
The case $a=-n$, $b=2x+n+1$ and $c=y+1$ of Lemma \ref{lemm-d} can be
manipulated as
 \bnm
\sum_{k=0}^n(-1)^kk^2\binm{n}{k}\frac{\binm{2x+n+k}{k}\binm{y+k}{k}}{\binm{x+k}{k}\binm{2y+k}{k}}
=\begin{cases}
 \gamma_m(x,y)
\frac{\binm{x-y+m}{m}\binm{m-\frac{1}{2}}{m}}{\binm{x+m}{m}\binm{y+m-\frac{1}{2}}{m}},&\qdn
n=2m;
\\[3mm]
\eta_m(x,y)
\frac{\binm{x-y+m}{m}\binm{m+\frac{1}{2}}{m}}{\binm{x+m}{m}\binm{y+m+\frac{1}{2}}{m}},&\qdn
n=2m+1,
\end{cases}
 \enm
where the corresponding expressions are
 \bnm
&&\gamma_m(x,y)=\frac{2m(2x+2m+1)\{4m^2+2(2x+1)m+xy-y^2-x-y\}}{(x-y+1)(x-y+2)},
\\
&&\eta_m(x,y)=\frac{2\{y^2-(m+1)(2m+1)-x(y+2m+1)\}}{(x-y+1)(2y+1)}
 \\&&\qquad\qquad+\:\frac{2m(2x+2m+3)\{3(y+2)(y-x-2)-2(m-1)(2x+2m+5)\}}{(x-y+1)(x-y+2)(2y+1)}.
 \enm
 Applying the derivative operator $\mathcal{D}_y$ to both sides of
 it, we get
 \bnm
&&\qqdn\xqdn\sum_{k=0}^n(-1)^k\binm{n}{k}\frac{\binm{2x+n+k}{k}\binm{y+k}{k}}{\binm{x+k}{k}\binm{2y+k}{k}}
k^2\big\{2H_k(2y)-H_k(y)\big\}\\
 &&\qqdn\xqdn\:=\:\begin{cases}
\mu_m(x,y)\frac{\binm{x-y+m}{m}\binm{m-\frac{1}{2}}{m}}{\binm{x+m}{m}\binm{y+m-\frac{1}{2}}{m}}
\big\{\sst H_{m-2}(x-y+2)+H_m(y-\frac{1}{2})-\nu_m(x,y)\big\},&\qdn
n=2m;
\\[3mm]
\lambda_m(x,y)\frac{\binm{x-y+m}{m}\binm{m+\frac{1}{2}}{m}}{\binm{x+m}{m}\binm{y+m+\frac{1}{2}}{m}}
\big\{\sst H_{m-2}(x-y+2)+H_{m+1}(y-\frac{1}{2})-\omega_m(x,y)\big\}
,&\qdn n=2m+1,
\end{cases}
 \enm
where the symbols on the right hand side stand for
 \bnm
&&\xxqdn\mu_m(x,y)=\frac{2m(2x+2m+1)\{x(y+4m-1)-y(y+1)+2m(2m+1)\}}{(x-y+1)(x-y+2)},\\
&&\xxqdn\nu_m(x,y)=\frac{x-2y-1}{x(y+4m-1)-y(y+1)+2m(2m+1)},\\
&&\xxqdn\lambda_m(x,y)=\frac{\begin{cases}
2m(2x+2m+3)\{3y^2-3xy-2(2m+1)(x+m+1)\}\\
+2(x-y+2)\{y^2-xy-(2m+1)(x+m+1)\}
\end{cases}}
{(x-y+1)(x-y+2)(2y+1)},\\
&&\xxqdn\omega_m(x,y)=\frac{m(3x-6y-1)(2x+2m+3)+x^2+x-1-4(x+1)y+3y^2}
{\begin{cases}
(2m+1)(x+m+1)(2+x+4mx+6m+4m^2)\\
-(1+3m+2m^2-x-x^2-6m^2x-7mx-6mx^2)y\\
-(2+9m+6m^2+2x+6mx)y^2+y^3
\end{cases}}.\\
 \enm
 Setting $y=0$ in the last
equation, we gain Theorem \ref{thm-c} to finish the proof.
\end{proof}

Choosing $x=p$ with $p\in \mathbb{N}_0$ in Theorem \ref{thm-c}, we
have the summation formula on harmonic numbers.
\begin{corl} \label{corl-c}
Let $p$ be a nonnegative integer. Then
 \bnm
&&\xqdn\xxqdn\sum_{k=0}^n(-1)^k\binm{n}{k}\binm{2p+n+k}{p+k}k^2H_{k}=\frac{n(n^2+n+2pn-p)}{p+2}\binm{2p+n+1}{p+1}\\
&&\xqdn\xxqdn\,\,\times\begin{cases}
2H_{2m}-H_m+H_{p+m}-H_{p+2}+\frac{1-p}{2m(2p+2m+1)-p},&\qdn n=2m;
\\[2mm]
H_m-2H_{2m}+H_{p+2}-H_{p+m+1}-\frac{p+3}{2m(2p+2m+3)+p+2},&\qdn
n=2m+1.
\end{cases}
 \enm
\end{corl}

When $x\to\infty$, Theorem \ref{thm-c} creates the concise harmonic
 number identity:
\bnm
  \sum_{k=0}^n(-2)^k\binm{n}{k}k^2H_{k}=2n(2n-1)\times
 \begin{cases}
 2H_{2m}-H_m-\frac{1}{4m-1},&\qdn n=2m;
\\[2mm]
H_m-2H_{2m}-\frac{1}{4m+1},&\qdn n=2m+1.
\end{cases}
 \enm

\section{The second family of summation formulae involving\\ generalized harmonic numbers}
\begin{lemm}\label{lemm-e}
Let $a$ and $b$ be both complex numbers. Then
 \bnm
 _3F_2\ffnk{cccc}{1}{a,b-a,-n}{\frac{b}{2},-2n}
=\frac{a}{b}\frac{(\frac{2+a}{2})_n(\frac{1+b-a}{2})_n}{(\frac{1}{2})_n(\frac{2+b}{2})_n}
+\frac{b-a}{b}\frac{(\frac{1+a}{2})_n(\frac{2+b-a}{2})_n}{(\frac{1}{2})_n(\frac{2+b}{2})_n}.
 \enm
\end{lemm}

\begin{proof}
 Perform the replacement $b\to 1+b-a$ in \eqref{watson-ter} to
 achieve
 \bnm
 _3F_2\ffnk{cccc}{1}{a,1+b-a,-n}{1+\frac{b}{2},-2n}
=\frac{(\frac{1+a}{2})_n(\frac{2+b-a}{2})_n}{(\frac{1}{2})_n(\frac{2+b}{2})_n}.
 \enm
Replace $a$ by $1+a$ in it to attain
 \bnm
 _3F_2\ffnk{cccc}{1}{1+a,b-a,-n}{1+\frac{b}{2},-2n}
=\frac{(\frac{2+a}{2})_n(\frac{1+b-a}{2})_n}{(\frac{1}{2})_n(\frac{2+b}{2})_n}.
 \enm
 The linear combination of the last two equations
gives Lemma \ref{lemm-e}.
\end{proof}

\begin{thm} \label{thm-d}
Let $x$ be a complex number. Then
 \bnm
&&\xxqdn\qqdn\sum_{k=0}^n\binm{2n-k}{n}\binm{x+k}{k}H_{k}^{\langle2\rangle}(x)=\frac{\binm{x+2n+1}{n}}{8}\\
&&\xxqdn\qqdn\,\,\times\:\Big\{4H_{2n}^{\langle2\rangle}(x+1)
-\big[H_n(\tfrac{x}{2})-H_n(\tfrac{x+1}{2})\big]\big[H_n(\tfrac{x}{2})-H_n(\tfrac{x+1}{2})-\tfrac{4}{x+1}\big]\Big\}.
 \enm
\end{thm}

\begin{proof}
The case $a=x+1$ and $b=y+2$ of Lemma \ref{lemm-e} reads as
 \bmn\label{watson-c}
\xqdn\xqdn\sum_{k=0}^n\frac{\binm{x+k}{k}\binm{y-x+k}{k}\binm{2n-k}{n}}{\binm{\frac{y}{2}+k}{k}}
&&\xqdn\!\!=\frac{4^n(x+1)}{y+2}\frac{\binm{\frac{x+1}{2}+n}{n}\binm{\frac{y-x}{2}+n}{n}}
 {\binm{\frac{y+2}{2}+n}{n}}
\nnm\\&&\xqdn\!
+\:\frac{4^n(y-x+1)}{y+2}\frac{\binm{\frac{x}{2}+n}{n}\binm{\frac{y-x+1}{2}+n}{n}}
 {\binm{\frac{y+2}{2}+n}{n}}.
 \emn
 Applying the derivative operator $\mathcal{D}_x$ to both sides of
 \eqref{watson-c}, we obtain
 \bnm
&&\xqdn\sum_{k=0}^n\frac{\binm{x+k}{k}\binm{y-x+k}{k}\binm{2n-k}{n}}{\binm{\frac{y}{2}+k}{k}}
\big\{H_k(x)-H_k(y-x)\big\}=\frac{4^n}{y+2}\frac{1}{\binm{\frac{y+2}{2}+n}{n}}\\
&&\xqdn\:\times\:
\bigg\{\binm{\frac{x+1}{2}+n}{n}\binm{\frac{y-x}{2}+n}{n}-\binm{\frac{x}{2}+n}{n}\binm{\frac{y-x+1}{2}+n}{n}
+\theta_n(x,y)\bigg\},
 \enm
where the corresponding expression is
 \bnm
\qqdn\theta_n(x,y)&&\xqdn\!=\frac{(x+1)\binm{\frac{x+1}{2}+n}{n}\binm{\frac{y-x}{2}+n}{n}\big[H_n(\tfrac{x+1}{2})-H_n(\tfrac{y-x}{2})\big]}{2}\\
&&\xqdn\!+\,\frac{(y-x+1)\binm{\frac{x}{2}+n}{n}\binm{\frac{y-x+1}{2}+n}{n}\big[H_n(\tfrac{x}{2})-H_n(\tfrac{y-x+1}{2})\big]}{2}.
 \enm
The last equation can be reformulated as
 \bnm
&&\xqdn\sum_{k=0}^n\frac{\binm{x+k}{k}\binm{y-x+k}{k}\binm{2n-k}{n}}{\binm{\frac{y}{2}+k}{k}}
\sum_{i=1}^k\frac{1}{(x+i)(y-x+i)}
  \\&&\xqdn\:=\:
\frac{4^n}{y+2}\frac{1}{\binm{\frac{y+2}{2}+n}{n}}
\bigg\{\frac{\binm{\frac{x+1}{2}+n}{n}\binm{\frac{y-x}{2}+n}{n}-\binm{\frac{x}{2}+n}{n}\binm{\frac{y-x+1}{2}+n}{n}}{y-2x}
+\frac{\theta_n(x,y)}{y-2x}\bigg\}.
 \enm
Finding the limit $y\to2x$ of it and using the two relations
 \bnm
&&\text{Lim}_{y\to2x}\frac{\binm{\frac{x+1}{2}+n}{n}\binm{\frac{y-x}{2}+n}{n}-\binm{\frac{x}{2}+n}{n}\binm{\frac{y-x+1}{2}+n}{n}}{y-2x}\\
&&\:\:=\:\text{Lim}_{y\to2x}\mathcal{D}_y\bigg\{\binm{\frac{x+1}{2}+n}{n}\binm{\frac{y-x}{2}+n}{n}
-\binm{\frac{x}{2}+n}{n}\binm{\frac{y-x+1}{2}+n}{n}\bigg\}\\
&&\:\:=\:\frac{1}{2}\binm{\frac{x}{2}+n}{n}\binm{\frac{x+1}{2}+n}{n}
\big[H_n(\tfrac{x}{2})-H_n(\tfrac{x+1}{2})\big],\\
&&\text{Lim}_{y\to2x}\frac{\theta_n(x,y)}{y-2x}\\
&&\:\:=\:\text{Lim}_{y\to2x}\mathcal{D}_y\theta_n(x,y)\\
&&\:\:=\:\frac{1}{2}\binm{\frac{x}{2}+n}{n}\binm{\frac{x+1}{2}+n}{n}
\big[H_n(\tfrac{x}{2})-H_n(\tfrac{x+1}{2})\big]\\
&&\:\:+\:\:\frac{x+1}{4}\binm{\frac{x}{2}+n}{n}\binm{\frac{x+1}{2}+n}{n}
\Big\{4H^{\langle2\rangle}_{2n}(x+1)-\big[H_n(\tfrac{x}{2})-H_n(\tfrac{x+1}{2})\big]^2\Big\}
 \enm
from L'H\^{o}spital rule, we get Theorem \ref{thm-d} to complete the
proof.
\end{proof}

Taking $x=p$ with $p\in \mathbb{N}_0$ in Theorem \ref{thm-d} and
utilizing \eqref{watson-c}, we deduce the summation formula on
harmonic numbers of 2-order.

\begin{corl} \label{corl-d}
Let $p$ be a nonnegative integer. Then
 \bnm
\sum_{k=0}^n\binm{2n-k}{n}\binm{p+k}{k}H_{p+k}^{\langle2\rangle}
=\frac{\binm{p+2n+1}{n}}{2}\Big\{H_{p+2n+1}^{\langle2\rangle}+H_{p}^{\langle2\rangle}-A_n(p)-\tfrac{1}{(p+1)^2}\Big\},
 \enm
where the symbol on the right hand side stands for
 \bnm
 A_n(p)=
 \begin{cases}
 (H_{2q+2n+1}-H_{q+n}-H_{2q+1}+H_q)\\\times(H_{2q+2n+1}-H_{q+n}-H_{2q+1}+H_q+\frac{2}{2q+1}),&\qdn p=2q;
\\[2mm]
(H_{2q+2n+2}-H_{q+n+1}-H_{2q+2}+H_{q+1})\\\times(H_{2q+2n+2}-H_{q+n+1}-H_{2q+2}+H_{q}),&\qdn
p=2q+1.
\end{cases}
 \enm
\end{corl}

When $p=0$, Corollary \ref{corl-d} reduces to the concise harmonic
 number identity:
\bnm
 \sum_{k=0}^n\binm{2n-k}{n}H_{k}^{\langle2\rangle}
=\frac{\binm{2n+1}{n}}{2}\big\{H_{2n+1}^{\langle2\rangle}-(H_{2n+1}-H_n)^2\big\}.
 \enm

\begin{lemm}\label{lemm-f}
Let $a$ and $b$ be both complex numbers. Then
 \bnm
\xxqdn
\sum_{k=0}^{n}k\frac{(a)_k(b-a)_k(-n)_k}{k!(\frac{b}{2})_k(-2n)_k}
&&\xqdn\!=\frac{a(b-a)(2a-b+2n)}{b(b+2)}\frac{(\frac{1+a}{2})_n(\frac{2+b-a}{2})_n}{(\frac{1}{2})_n(\frac{4+b}{2})_n}\\
&&\xqdn\!+\:\frac{a(b-a)(b-2a+2n)}{b(b+2)}\frac{(\frac{2+a}{2})_n(\frac{1+b-a}{2})_n}{(\frac{1}{2})_n(\frac{4+b}{2})_n}.
 \enm
\end{lemm}

\begin{proof}
 For two complex variables $x$ and
$q$, define the $q$-shifted factorial to be
 \[(x;q)_0=1\quad\text{and}\quad (x;q)_n=\prod_{k=0}^{n-1}(1-xq^k)
  \quad\text{with}\quad n\in \mathbb{N}.\]
Following Gasper and Rahman \cito{gasper}, define the basic
hypergeometric series by
\[_{r}\phi_s\ffnk{cccc}{q;z}{c_1,c_2,\cdots,c_r}{d_1,d_2,\cdots,d_s}
 =\sum_{k=0}^\infty\ffnk{ccccc}{q}{c_1,c_2,\cdots,c_r}{d_1,d_2,\cdots,d_s}_k
  \bigg\{(-1)^kq^{\binm{k}{2}}\bigg\}^{1+s-r}\qdn\frac{z^k}{(q;q)_k},\]
where $\{c_i\}_{i\geq0}$ and $\{d_j\}_{j\geq1}$ are complex
parameters such that no zero factors appear in the denominators of
the summand on the right hand side. Employing the substitutions
$a\to q^a$, $b\to q^{1+b}$ in the equation
 \bnm
{_4\phi_3}\ffnk{ccccc}{q;q}{a,b/a,q^{-n},-q^{-n-1}}
{\sqrt{qb},-\sqrt{qb},q^{-2n-1}}=
\frac{(qa;q^2)_{n+1}(qb/a;q^2)_{n+1}}{(q;q^2)_{n+1}(qb;q^2)_{n+1}}
-q^{n+1}\frac{(a;q^2)_{n+1}(b/a;q^2)_{n+1}}{(q;q^2)_{n+1}(qb;q^2)_{n+1}}
 \enm
 due to Wei et al. \citu{wei-b}{Corollary 11} and then letting $q\to1$,
we gain
 \bnm
 _3F_2\ffnk{cccc}{1}{a,1+b-a,-n}{\frac{2+b}{2},-2n-1}
=\frac{(\frac{1+a}{2})_{n+1}(\frac{2+b-a}{2})_{n+1}}{(\frac{1}{2})_{n+1}(\frac{2+b}{2})_{n+1}}
-\frac{(\frac{a}{2})_{n+1}(\frac{1+b-a}{2})_{n+1}}{(\frac{1}{2})_{n+1}(\frac{2+b}{2})_{n+1}}.
 \enm
Substitute $1+a$ for $a$ in it to achieve
 \bnm
 _3F_2\ffnk{cccc}{1}{1+a,b-a,-n}{\frac{2+b}{2},-2n-1}
=\frac{(\frac{2+a}{2})_{n+1}(\frac{1+b-a}{2})_{n+1}}{(\frac{1}{2})_{n+1}(\frac{2+b}{2})_{n+1}}
-\frac{(\frac{1+a}{2})_{n+1}(\frac{b-a}{2})_{n+1}}{(\frac{1}{2})_{n+1}(\frac{2+b}{2})_{n+1}}.
 \enm
The linear combination of the last two equations provides
 \bmn\label{watson-d}
 \xxqdn_3F_2\ffnk{cccc}{1}{a,b-a,-n}{\frac{b}{2},-2n-1}
&&\xqdn\!=\frac{2a-b+2+2n}{b}\frac{(\frac{a}{2})_{n+1}(\frac{1+b-a}{2})_{n+1}}{(\frac{1}{2})_{n+1}(\frac{2+b}{2})_{n+1}}
\nnm\\
&&\xqdn\!+\,\frac{b-2a+2+2n}{b}\frac{(\frac{1+a}{2})_{n+1}(\frac{b-a}{2})_{n+1}}{(\frac{1}{2})_{n+1}(\frac{2+b}{2})_{n+1}}.
 \emn
It is not difficult to verify that
 \bnm
 \sum_{k=0}^{n}k\frac{(a)_k(b-a)_k(-n)_k}{k!(\frac{b}{2})_k(-2n)_k}
 &&\xqdn\!=\sum_{k=1}^{n}\frac{(a)_k(b-a)_k(-n)_k}{(k-1)!(\frac{b}{2})_k(-2n)_k}\\
 &&\xqdn\!=\sum_{k=0}^{n-1}\frac{(a)_{k+1}(b-a)_{k+1}(-n)_{k+1}}{k!(\frac{b}{2})_{k+1}(-2n)_{k+1}}\\
&&\xqdn\!=\frac{a(b-a)}{b}{_3F_2}\ffnk{cccc}{1}{1+a,1+b-a,1-n}{\frac{2+b}{2},1-2n}.
 \enm
 Computing the
series on the right hand side by \eqref{watson-d}, we attain Lemma
\ref{lemm-f} to finish the proof.
\end{proof}

\begin{thm} \label{thm-e}
Let $x$ be a complex number. Then
 \bnm
&&\xxqdn\sum_{k=0}^n\binm{2n-k}{n}\binm{x+k}{k}kH_{k}^{\langle2\rangle}(x)=\frac{x+1}{8}\binm{x+2n+1}{n-1}\\
&&\xxqdn\,\,\times\:\Big\{4H_{2n+1}^{\langle2\rangle}(x)
-\big[H_n(\tfrac{x}{2})-H_n(\tfrac{x+1}{2})\big]\big[H_n(\tfrac{x}{2})-H_n(\tfrac{x+1}{2})+\tfrac{4}{n}\big]+\tfrac{4}{(x+1)^2}\Big\}.
 \enm
\end{thm}

\begin{proof}
The case $a=x+1$ and $b=y+2$ of Lemma \ref{lemm-f} can be written as
 \bmn\label{watson-e}
\xqdn\sum_{k=0}^nk\frac{\binm{x+k}{k}\binm{y-x+k}{k}\binm{2n-k}{n}}{\binm{\frac{y}{2}+k}{k}}
&&\xqdn\!\!=\frac{4^n(x+1)(y-x+1)(2x-y+2n)}{(y+2)(y+4)}
 \frac{\binm{\frac{x}{2}+n}{n}\binm{\frac{y-x+1}{2}+n}{n}}{\binm{\frac{y+4}{2}+n}{n}}
  \nnm\\&&\xqdn\!
+\:\frac{4^n(x+1)(y-x+1)(y-2x+2n)}{(y+2)(y+4)}
 \frac{\binm{\frac{x+1}{2}+n}{n}\binm{\frac{y-x}{2}+n}{n}}{\binm{\frac{y+4}{2}+n}{n}}.
 \emn
 Applying the derivative operator $\mathcal{D}_x$ to both sides of
 \eqref{watson-e}, we obtain
 \bnm
&&\xqdn\sum_{k=0}^n\frac{\binm{x+k}{k}\binm{y-x+k}{k}\binm{2n-k}{n}}{\binm{\frac{y}{2}+k}{k}}
k\big\{H_k(x)-H_k(y-x)\big\}\\
&&\xqdn\:=\:
\frac{4^n}{(y+2)(y+4)}\frac{\varepsilon_n(x,y)}{\binm{\frac{y+4}{2}+n}{n}}+
\frac{4^{n}(x+1)(y-x+1)}{(y+2)(y+4)}\frac{\zeta_n(x,y)}{\binm{\frac{y+4}{2}+n}{n}},
 \enm
where the corresponding expressions are
 \bnm
\varepsilon_n(x,y)&&\xqdn\!=\{2-6x^2+6xy-(y-2)y-2n(2x-y)\}\binm{\frac{x}{2}+n}{n}\binm{\frac{y-x+1}{2}+n}{n}\\
&&\xqdn\!-\,\{2-6x^2+6xy-(y-2)y+2n(2x-y)\}\binm{\frac{x+1}{2}+n}{n}\binm{\frac{y-x}{2}+n}{n},\\
\zeta_n(x,y)&&\xqdn\!=\frac{(2x-y+2n)\binm{\frac{x}{2}+n}{n}\binm{\frac{y-x+1}{2}+n}{n}\big[H_n(\tfrac{x}{2})-H_n(\tfrac{y-x+1}{2})\big]}{2}\\
&&\xqdn\!+\,\frac{(y-2x+2n)\binm{\frac{x+1}{2}+n}{n}\binm{\frac{y-x}{2}+n}{n}\big[H_n(\tfrac{x+1}{2})-H_n(\tfrac{y-x}{2})\big]}{2}.
 \enm
The last equation can be reformulated as
 \bnm
&&\xqdn\sum_{k=0}^n\frac{\binm{x+k}{k}\binm{y-x+k}{k}\binm{2n-k}{n}}{\binm{\frac{y}{2}+k}{k}}
k\sum_{i=1}^k\frac{1}{(x+i)(y-x+i)}
 \nnm\\&&\xqdn\:=\:
\frac{4^n}{(y+2)(y+4)}\frac{1}{\binm{\frac{y+4}{2}+n}{n}}\frac{\varepsilon_n(x,y)}{y-2x}+
\frac{4^{n}(x+1)(y-x+1)}{(y+2)(y+4)}\frac{1}{\binm{\frac{y+4}{2}+n}{n}}\frac{\zeta_n(x,y)}{y-2x}.
 \enm
Finding the limit $y\to2x$ of it and exploiting the two relations
 \bnm
&&\text{Lim}_{y\to2x}\frac{\varepsilon_n(x,y)}{y-2x}\\
&&\:\:=\:\text{Lim}_{y\to2x}\mathcal{D}_y\varepsilon_n(x,y)\\
&&\:\:=\:(x+1)^2\binm{\frac{x}{2}+n}{n}\binm{\frac{x+1}{2}+n}{n}
\big[H_n(\tfrac{x+1}{2})-H_n(\tfrac{x}{2})+\tfrac{4n}{(x+1)^2}\big],\\
&&\text{Lim}_{y\to2x}\frac{\zeta_n(x,y)}{y-2x}\\
&&\:\:=\:\text{Lim}_{y\to2x}\mathcal{D}_y\zeta_n(x,y)\\
&&\:\:=\:\frac{n}{2}\binm{\frac{x}{2}+n}{n}\binm{\frac{x+1}{2}+n}{n}\\
&&\:\:\times\:\,\Big\{4H_{2n}^{\langle2\rangle}(x+1)
-\big[H_n(\tfrac{x}{2})-H_n(\tfrac{x+1}{2})\big]\big[H_n(\tfrac{x}{2})-H_n(\tfrac{x+1}{2})+\tfrac{2}{n}\big]\Big\}
 \enm
from L'H\^{o}spital rule, we get Theorem \ref{thm-e} to complete the
proof.
\end{proof}

Selecting $x=p$ with $p\in \mathbb{N}_0$ in Theorem \ref{thm-e} and
using \eqref{watson-e}, we have the summation formula on harmonic
numbers of 2-order.

\begin{corl} \label{corl-e}
Let $p$ be a nonnegative integer. Then
 \bnm
\quad\sum_{k=0}^n\binm{2n-k}{n}\binm{p+k}{k}kH_{p+k}^{\langle2\rangle}
=\frac{p+1}{2}\binm{p+2n+1}{n-1}\Big\{H_{p+2n+1}^{\langle2\rangle}+H_{p+1}^{\langle2\rangle}-B_n(p)\Big\},
 \enm
where the symbols on the right hand side stand for
 \bnm
 B_n(p)=
 \begin{cases}
 (H_{2q+2n+1}-H_{q+n}-H_{2q+1}+H_q)\\\times(H_{2q+2n+1}-H_{q+n}-H_{2q+1}+H_q-\frac{2}{n}),&\qdn p=2q;
\\[2mm]
(H_{2q+2n+2}-H_{q+n+1}-H_{2q+2}+H_{q+1})\\\times(H_{2q+2n+2}-H_{q+n+1}-H_{2q+2}+H_{q+1}+\frac{2}{n}),&\qdn
p=2q+1.
\end{cases}
 \enm
\end{corl}

When $p=0$, Corollary \ref{corl-e} produces the concise harmonic
 number identity:
\bnm
 \quad\sum_{k=0}^n\binm{2n-k}{n}kH_{k}^{\langle2\rangle}
=\frac{\binm{2n+1}{n-1}}{2}\big\{H_{2n+1}^{\langle2\rangle}-\tfrac{2}{n}-(H_{2n+1}-H_n)(H_{2n+1}-H_n-2-\tfrac{2}{n})\big\}.
 \enm

\begin{lemm}\label{lemm-g}
Let $a$ and $b$ be both complex numbers. Then
 \bnm
 \quad\sum_{k=0}^{n}k^2\frac{(a)_k(b-a)_k(-n)_k}{k!(\frac{b}{2})_k(-2n)_k}
=\Theta_n(a,b)\frac{(\frac{1+a}{2})_n(\frac{b-a}{2})_{n+1}}{(\frac{1}{2})_n(\frac{4+b}{2})_{n+1}}
+\Omega_n(a,b)\frac{(\frac{a}{2})_{n+1}(\frac{1+b-a}{2})_n}{(\frac{1}{2})_n(\frac{4+b}{2})_{n+1}}.
 \enm
 where the corresponding expressions are
\bnm
 &&\xxqdn\Theta_n(a,b)=\frac{\sst2a^3(2a-3b+2n)+2a^2(b^2-b-3bn-2n^2+2n+2)+a(2n+1)(b^2-2b+2bn+4n)}{b(b+2)},\\
  &&\xxqdn\Omega_n(a,b)=\frac{\sst(a-b)^3(2a+1)+2(a-b)^2(a^2-an+n+1)-(a-b)(2n+1)(a^2-2a+2an+4n)}{b(b+2)}.
 \enm
\end{lemm}

\begin{proof}
 Performing the replacements $a\to q^a$, $b\to q^{1+b}$ in the equation
 \bnm
&&\xqdn{_4\phi_3}\ffnk{ccccc}{q;q}{a,b/a,q^{-n},-q^{-n-2}}
{\sqrt{qb},-\sqrt{qb},q^{-2n-2}}
=\frac{(1+q)(1-q^{2n+3})}{q^{n+1}-q^{-n-1}}\frac{(a;q^2)_{n+2}(b/a;q^2)_{n+2}}{(q;q^2)_{n+2}(qb;q^2)_{n+2}}\\
&&\xqdn\:\:+\:\bigg\{1-\frac{(1-a)(1-b/a)(1-q^{2n+4})}{(1-q^{-2n-2})(1-aq^{2n+3})(1-bq^{2n+3}/a)}\bigg\}
\frac{(qa;q^2)_{n+2}(qb/a;q^2)_{n+2}}{(q;q^2)_{n+2}(qb;q^2)_{n+2}}.
 \enm
due to Wei et al. \citu{wei-b}{Corollary 12} and then letting
$q\to1$, we gain
  \bnm
 &&\qqdn\xqdn_3F_2\ffnk{cccc}{1}{a,1+b-a,-n}{\frac{2+b}{2},-2n-2}
=-\frac{2n+3}{n+1}\frac{(\frac{a}{2})_{n+2}(\frac{1+b-a}{2})_{n+2}}{(\frac{1}{2})_{n+2}(\frac{2+b}{2})_{n+2}}\\
&&\qqdn\xqdn\:+\:\:\frac{a(1+b-a)(n+2)+(n+1)(a+2n+3)(b-a+2n+4)}{(n+1)(a+2n+3)(b-a+2n+4)}
\frac{(\frac{1+a}{2})_{n+2}(\frac{2+b-a}{2})_{n+2}}{(\frac{1}{2})_{n+2}(\frac{2+b}{2})_{n+2}}.
 \enm
Replace $a$ by $1+a$ in it to achieve
 \bnm
 &&\xqdn_3F_2\ffnk{cccc}{1}{1+a,b-a,-n}{\frac{2+b}{2},-2n-2}
=-\frac{2n+3}{n+1}\frac{(\frac{1+a}{2})_{n+2}(\frac{b-a}{2})_{n+2}}{(\frac{1}{2})_{n+2}(\frac{2+b}{2})_{n+2}}\\
&&\xqdn\:+\:\:\frac{(1+a)(b-a)(n+2)+(n+1)(a+2n+4)(b-a+2n+3)}{(n+1)(a+2n+4)(b-a+2n+3)}
\frac{(\frac{2+a}{2})_{n+2}(\frac{1+b-a}{2})_{n+2}}{(\frac{1}{2})_{n+2}(\frac{2+b}{2})_{n+2}}.
 \enm
The linear combination of the last two equations offers
 \bmn\label{watson-f}
\qqdn\xqdn _3F_2\ffnk{cccc}{1}{a,b-a,-n}{\frac{b}{2},-2n-2}
&&\xqdn\!=\frac{\sst a(2a-b+n+1)-(n+1)(b-a+2n+4)}{(n+1)(a-1)b}
\frac{(\frac{a-1}{2})_{n+2}(\frac{b-a}{2})_{n+2}}{(-\frac{1}{2})_{n+2}(\frac{2+b}{2})_{n+2}}
\nnm\\
&&\xqdn\!+\,\frac{\sst (b-a)(b-2a+n+1)-(n+1)(a+2n+4)}{(n+1)(b-a-1)b}
\frac{(\frac{a}{2})_{n+2}(\frac{b-a-1}{2})_{n+2}}{(-\frac{1}{2})_{n+2}(\frac{2+b}{2})_{n+2}}.
 \emn
It is routine to show that
 \bnm
 &&\xqdn\sum_{k=0}^{n}k^2\frac{(a)_k(b-a)_k(-n)_k}{k!(\frac{b}{2})_k(-2n)_k}
 =\sum_{k=1}^{n}k\frac{(a)_k(b-a)_k(-n)_k}{(k-1)!(\frac{b}{2})_k(-2n)_k}\\
 &&\qdn\qqdn=\sum_{k=0}^{n-1}(k+1)\frac{(a)_{k+1}(b-a)_{k+1}(-n)_{k+1}}{k!(\frac{b}{2})_{k+1}(-2n)_{k+1}}\\
&&\qdn\qqdn=\frac{a(b-a)}{b}{_3F_2}\ffnk{cccc}{1}{1+a,1+b-a,1-n}{\frac{2+b}{2},1-2n}\\
 &&\qdn\qqdn+\:\frac{a(b-a)}{b}
  \sum_{k=1}^{n-1}\frac{(1+a)_{k}(1+b-a)_{k}(1-n)_{k}}{(k-1)!(\frac{2+b}{2})_{k}(1-2n)_{k}}\\
&&\qdn\qqdn=\frac{a(b-a)}{b}{_3F_2}\ffnk{cccc}{1}{1+a,1+b-a,1-n}{\frac{2+b}{2},1-2n}\\
&&\qdn\qqdn+\:\frac{2a(1+a)(b-a)(1+b-a)(1-n)}{b(2+b)(1-2n)}{_3F_2}\ffnk{cccc}{1}{2+a,2+b-a,2-n}{\frac{4+b}{2},2-2n}.
 \enm
 Evaluating the two
series on the right hand side by \eqref{watson-d} and
\eqref{watson-f}, we attain Lemma \ref{lemm-g} to finish the proof.
\end{proof}

\begin{thm} \label{thm-f}
Let $x$ be a complex number. Then
 \bnm
&&\xxqdn\sum_{k=0}^n\binm{2n-k}{n}\binm{x+k}{k}k^2H_{k}^{\langle2\rangle}(x)=\frac{(x+1)(xn+3n+1)}{8(n-1)}\binm{x+2n+1}{n-2}\\
&&\xxqdn\,\,\times\:\Big\{4H_{2n+1}^{\langle2\rangle}(x)+\tfrac{8(x^2+3x+n^2-n+1)}{n(x+1)(xn+3n+1)}
-\big[H_n(\tfrac{x}{2})-H_{n+1}(\tfrac{x-1}{2})\big]\\
&&\xqdn\,\,\times\:
\big[H_n(\tfrac{x}{2})-H_{n+1}(\tfrac{x-1}{2})-\tfrac{4(x+n+1)(x^2+3x-2n+1)}{n(x+1)(xn+3n+1)}\big]\Big\}.
 \enm
\end{thm}

\begin{proof}
The case $a=x+1$ and $b=y+2$ of Lemma \ref{lemm-g} can be
manipulated as
 \bmn\label{watson-g}
\sum_{k=0}^nk^2\frac{\binm{x+k}{k}\binm{y-x+k}{k}\binm{2n-k}{n}}{\binm{\frac{y}{2}+k}{k}}=\delta_n(x,y)
\frac{\binm{\frac{x+1}{2}+n}{n+1}\binm{\frac{y-x}{2}+n}{n}}{\binm{\frac{y+6}{2}+n}{n+1}}+
 \xi_n(x,y)\frac{\binm{\frac{x}{2}+n}{n}\binm{\frac{y-x+1}{2}+n}{n+1}}{\binm{\frac{y+6}{2}+n}{n+1}},
 \emn
where the symbols on the right hand side stand for
 \bnm
\delta_n(x,y)&&\xqdn\!=\frac{4^n(x-y-1)}{y+2}\frac{\sst2x^2(2x-3y-2n)+2x(y^2-3y+yn-2-2n^2)+3y^2+2y-4n-12n^2}{y+4},\\
\xi_n(x,y)&&\xqdn\!=\frac{4^n(x+1)}{y+2}\frac{\sst2x^2(2x-3y+2n)+2x(y^2-3y-3yn-2-2n^2)+(2n+3)y^2+2(2n^2+1)y+4n+12n^2}{y+4}.
 \enm
 Applying the derivative operator $\mathcal{D}_x$ to both sides of
 it, we obtain
 \bnm
&&\xqdn\sum_{k=0}^n\frac{\binm{x+k}{k}\binm{y-x+k}{k}\binm{2n-k}{n}}{\binm{\frac{y}{2}+k}{k}}
k^2\big\{H_k(x)-H_k(y-x)\big\}=\frac{\varphi_n(x,y)+\psi_n(x,y)}{\binm{\frac{y+6}{2}+n}{n+1}},
 \enm
where the the corresponding expressions are
 \bnm
\varphi_n(x,y)&&\xqdn\!= \mathcal{D}_x[\delta_n(x,y)]
\binm{\frac{x+1}{2}+n}{n+1}\binm{\frac{y-x}{2}+n}{n}+
 {D}_x[\xi_n(x,y)]\binm{\frac{x}{2}+n}{n}\binm{\frac{y-x+1}{2}+n}{n+1},\\
\psi_n(x,y)&&\xqdn\!=\frac{\delta_n(x,y)\binm{\frac{x+1}{2}+n}{n+1}
\binm{\frac{y-x}{2}+n}{n}\big[H_{n+1}(\tfrac{x-1}{2})-H_n(\tfrac{y-x}{2})\big]}{2}\\
&&\xqdn\!+\,\frac{\xi_n(x,y)\binm{\frac{x}{2}+n}{n}
\binm{\frac{y-x+1}{2}+n}{n+1}\big[H_n(\tfrac{x}{2})-H_{n+1}(\tfrac{y-x-1}{2})\big]}{2}.
 \enm
The last equation can be reformulated as
 \bnm
&&\xqdn\sum_{k=0}^n\frac{\binm{x+k}{k}\binm{y-x+k}{k}\binm{2n-k}{n}}{\binm{\frac{y}{2}+k}{k}}
k^2\sum_{i=1}^k\frac{1}{(x+i)(y-x+i)}
 \\&&\xqdn\:=\:
\frac{1}{\binm{\frac{y+6}{2}+n}{n+1}}\frac{\varphi_n(x,y)}{y-2x}+
\frac{1}{\binm{\frac{y+6}{2}+n}{n+1}}\frac{\psi_n(x,y)}{y-2x}.
 \enm
Fining the limit $y\to2x$ of it and utilizing the two relations
 \bnm
&&\xqdn\text{Lim}_{y\to2x}\frac{\varphi_n(x,y)}{y-2x}\\
&&\xqdn\:\:=\:\text{Lim}_{y\to2x}\mathcal{D}_y\varphi_n(x,y)\\
&&\xqdn\:\:=\:\frac{4^n(x+n+1)(x^2+3x-2n+1)}{2(x+1)(x+2)}\binm{\frac{x}{2}+n}{n}\binm{\frac{x+1}{2}+n}{n+1}\\
&&\xqdn\:\:\times\:\:\big[H_n(\tfrac{x}{2})-H_{n+1}(\tfrac{x-1}{2})
+\tfrac{4(x^2+3x+n^2-n+1)}{(x+n+1)(x^2+3x-2n+1)}\big],\\
&&\xqdn\text{Lim}_{y\to2x}\frac{\psi_n(x,y)}{y-2x}\\
&&\xqdn\:\:=\:\text{Lim}_{y\to2x}\mathcal{D}_y\psi_n(x,y)\\
&&\xqdn\:\:=\:\frac{4^{n-1}n(xn+3n+1)}{x+2}\binm{\frac{x}{2}+n}{n}\binm{\frac{x+1}{2}+n}{n+1}\Big\{4H_{2n+1}^{\langle2\rangle}(x)\\
&&\xqdn\:\:-\:\,
\big[H_n(\tfrac{x}{2})-H_{n+1}(\tfrac{x-1}{2})\big]\big[H_n(\tfrac{x}{2})-H_{n+1}(\tfrac{x-1}{2})-
\tfrac{2(x+n+1)(x^2+3x-2n+1)}{n(x+1)(xn+3n+1)}\big]\Big\}
 \enm
from L'H\^{o}spital rule, we get Theorem \ref{thm-f} to complete the
proof.
\end{proof}

Fixing $x=p$ with $p\in \mathbb{N}_0$ in Theorem \ref{thm-f} and
exploiting \eqref{watson-g}, we deduce the summation formula on
harmonic numbers of 2-order.

\begin{corl} \label{corl-f}
Let $p$ be a nonnegative integer. Then
 \bnm
&&\sum_{k=0}^n\binm{2n-k}{n}\binm{p+k}{k}k^2H_{p+k}^{\langle2\rangle}
=\frac{(p+1)(pn+3n+1)}{2(n-1)}\binm{p+2n+1}{n-2}\\
&&\,\,\times\:\Big\{H_{p+2n+1}^{\langle2\rangle}+H_{p}^{\langle2\rangle}-C_n(p)
+\tfrac{2(p^2+3p+n^2-n+1)}{n(p+1)(pn+3n+1)}\Big\},
 \enm
where the symbol on the right hand side stands for
 \bnm
 C_n(p)=
 \begin{cases}
 \big[H_{2q+2n+1}-H_{q+n}-H_{2q}+H_q\big]\\
 \times\big[H_{2q+2n+1}-H_{q+n}-H_{2q}+H_q+\frac{2(2q+n+1)(4q^2+6q-2n+1)}{n(2q+1)(2qn+3n+1)}\big],&\qqdn p=2q;
\\[2mm]
\big[H_{2q+2n+2}-H_{q+n+1}-H_{2q+1}+H_{q}\big]\\
\times\big[H_{2q+2n+2}-H_{q+n+1}-H_{2q+1}+H_{q}-\frac{(2q+n+2)(4q^2+10q-2n+5)}{n(q+1)(2qn+4n+1)}\big],&\qqdn
p=2q+1.
\end{cases}
 \enm
\end{corl}

When $p=0$, Corollary \ref{corl-f} creates the concise harmonic
 number identity:
\bnm
 &&\sum_{k=0}^n\binm{2n-k}{n}k^2H_{k}^{\langle2\rangle}=\frac{3n+1}{2(n-1)}\binm{2n+1}{n-2}\\
&&\,\,\times\:\Big\{H_{2n+1}^{\langle2\rangle}-\big[H_{2n+1}-H_n\big]
\big[H_{2n+1}-H_n-\tfrac{2(n+1)(2n-1)}{n(3n+1)}\big]+\tfrac{2(n^2-n+1)}{n(3n+1)}\Big\}.
 \enm

\section{The third family of summation formulae involving\\ generalized harmonic numbers}

\begin{thm} \label{thm-g}
Let $x$ be a complex number. Then
 \bnm
&&\xxqdn\sum_{k=0}^n\binm{2n-k}{n}\binm{x+k}{k}H^2_{k}(x)=\frac{\binm{x+2n+1}{n}}{8}
 \Big\{4H_{n}^{\langle2\rangle}(x+1)-4H_{n}^{\langle2\rangle}(x+n+1)\\
&&+8H^2_n(x+n+1)
-\big[H_n(\tfrac{x}{2})-H_n(\tfrac{x+1}{2})\big]\big[H_n(\tfrac{x}{2})-H_n(\tfrac{x+1}{2})-\tfrac{4}{x+1}\big]\Big\}.
 \enm
\end{thm}

\begin{proof}
The case $y=2x$ of \eqref{watson-c} reads as
 \bmn\label{watson-h}
\sum_{k=0}^n\binm{2n-k}{n}\binm{x+k}{k}=\binm{x+2n+1}{n}.
 \emn
 Applying the derivative operator $\mathcal{D}_x$ to both sides of
 \eqref{watson-h}, we gain
 \bmn\label{harmonic-a}
\sum_{k=0}^n\binm{2n-k}{n}\binm{x+k}{k}H_k(x)=\binm{x+2n+1}{n}\big\{H_{2n+1}(x)-H_{n+1}(x)\big\}.
 \emn
 Applying the derivative operator $\mathcal{D}_x$ to both sides of
 \eqref{harmonic-a}, we achieve
 \bnm
&&\sum_{k=0}^n\binm{2n-k}{n}\binm{x+k}{k}\big\{H^2_k(x)-H_{k}^{\langle2\rangle}(x)\big\}
 \nnm\\&&\:\:=\:\binm{x+2n+1}{n}\big\{[H_{2n+1}(x)-H_{n+1}(x)]^2
-[H_{2n+1}^{\langle2\rangle}(x)-H_{n+1}^{\langle2\rangle}(x)]\big\}.
 \enm
The sum of Theorem \ref{thm-d} and the last equation gives Theorem
\ref{thm-g}.
\end{proof}

Selecting $x=p$ with $p\in \mathbb{N}_0$ in Theorem \ref{thm-g} and
using \eqref{watson-h} and \eqref{harmonic-a}, we have the summation
formula on harmonic numbers.

\begin{corl} \label{corl-g}
Let $p$ be a nonnegative integer. Then
 \bnm
&&\xqdn\sum_{k=0}^n\binm{2n-k}{n}\binm{p+k}{k}H^2_{p+k}
=\frac{\binm{p+2n+1}{n}}{2}\Big\{2H_{p+n+1}^{\langle2\rangle}-H_{p+2n+1}^{\langle2\rangle}-H_{p+1}^{\langle2\rangle}\\
&&\xqdn\,\,+\:2(H_{p+n+1}-H_{p+2n+1})^2-2H_{p}(2H_{p+n+1}-2H_{p+2n+1}-H_{p})-A_n(p)\Big\},
 \enm
where $A_n(p)$ has appeared in Corollary \ref{corl-d}.
\end{corl}

When $p=0$, Corollary \ref{corl-g} reduces to the concise harmonic
 number identity:
\bnm
 &&\:\sum_{k=0}^{n}\binm{2n-k}{n}H^2_{k}=\frac{\binm{2n+1}{n}}{2}
\big\{2H_{n+1}^{\langle2\rangle}-H_{2n+1}^{\langle2\rangle}+2H^2_{n+1}-H^2_n
\\&&\qquad\qquad\qquad\qquad\qquad\qquad
+H_{2n+1}(H_{2n+1}-4H_{n+1}+2H_n)\big\}.
 \enm

\begin{thm} \label{thm-h}
Let $x$ be a complex number. Then
 \bnm
\sum_{k=0}^n\binm{2n-k}{n}\binm{x+k}{k}kH^2_{k}(x)&&\xqdn\!\!=\frac{x+1}{8}\binm{x+2n+1}{n-1}
 \Big\{8H_{n+2}^{\langle2\rangle}(x)-4H_{2n}^{\langle2\rangle}(x+1)\\
&&\xxqdn\xqdn+8\big[H_{n+2}(x)-H_{2n+1}(x)\big]\big[H_{n+2}(x)-H_{2n+1}(x)-\tfrac{2}{x+1}\big]\\
&&\xxqdn\xqdn-\big[H_n(\tfrac{x}{2})-H_n(\tfrac{x+1}{2})\big]\big[H_n(\tfrac{x}{2})-H_n(\tfrac{x+1}{2})+\tfrac{4}{n}\big]\Big\}.
 \enm
\end{thm}

\begin{proof}
The case $y=2x$ of \eqref{watson-e} can be written as
 \bmn\label{watson-i}
\sum_{k=0}^nk\binm{2n-k}{n}\binm{x+k}{k}=(x+1)\binm{x+2n+1}{n-1}.
 \emn
 Applying the derivative operator $\mathcal{D}_x$ to both sides of
 \eqref{watson-i}, we attain
 \bmn\label{harmonic-b}
\sum_{k=0}^n\binm{2n-k}{n}\binm{x+k}{k}kH_k(x)=(x+1)\binm{x+2n+1}{n-1}\big\{H_{2n+1}(x)-H_{n+1}(x+1)\big\}.
 \emn
 Applying the derivative operator $\mathcal{D}_x$ to both sides of
 \eqref{harmonic-b}, we obtain
 \bnm
&&\xxqdn\qqdn\sum_{k=0}^n\binm{2n-k}{n}\binm{x+k}{k}k\big\{H^2_k(x)-H_{k}^{\langle2\rangle}(x)\big\}=(x+1)\binm{x+2n+1}{n-1}
 \\&&\xxqdn\qqdn\:\:\times\:\Big\{[H_{n+2}^{\langle2\rangle}(x)-H_{2n+1}^{\langle2\rangle}(x)]+
 [H_{n+2}(x)-H_{2n+1}(x)][H_{n+2}(x)-H_{2n+1}(x)-\tfrac{2}{x+1}]\Big\}.
 \enm
The sum of Theorem \ref{thm-e} and the last equation provides
Theorem \ref{thm-h}.
\end{proof}

Choosing $x=p$ with $p\in \mathbb{N}_0$ in Theorem \ref{thm-h} and
utilizing \eqref{watson-i} and \eqref{harmonic-b}, we deduce the
summation formula on harmonic numbers.

\begin{corl} \label{corl-h}
Let $p$ be a nonnegative integer. Then
 \bnm
&&\xxqdn\sum_{k=0}^n\binm{2n-k}{n}\binm{p+k}{k}kH^2_{p+k}
=\frac{p+1}{2}\binm{p+2n+1}{n-1}\Big\{2H_{p+n+2}^{\langle2\rangle}-H_{p+2n+1}^{\langle2\rangle}-H_{p+1}^{\langle2\rangle}\\
&&\xqdn\,\,+\:2(H_{p+n+2}-H_{p+2n+1})^2-2H_{p+1}(2H_{p+n+2}-2H_{p+2n+1}-H_{p+1})-B_n(p)\Big\},
 \enm
where $B_n(p)$ has emerged in Corollary \ref{corl-e}.
\end{corl}

When $p=0$, Corollary \ref{corl-h} produces the concise harmonic
 number identity:
\bnm
 \sum_{k=0}^{n}\binm{2n-k}{n}kH^2_{k}=\frac{\binm{2n+1}{n-1}}{2}
\Big\{2H_{n+2}^{\langle2\rangle}-H_{2n+1}^{\langle2\rangle}+\big[H_{2n+1}-H_{n+2}\big]
\\\times\:\big[H_{2n+1}-H_{n+2}+\tfrac{2(3n^3+8n^2+6n+2)}{n(n+1)(n+2)}\big]
+\tfrac{2n^4+6n^3+6n^2+5n+4}{n(n+1)^2(n+2)^2}\Big\}.
 \enm

\begin{thm} \label{thm-i}
Let $x$ be a complex number. Then
 \bnm
&&\sum_{k=0}^n\binm{2n-k}{n}\binm{x+k}{k}k^2H^2_{k}(x)=\frac{(x+1)(xn+3n+1)}{8(n-1)}\binm{x+2n+1}{n-2}\\
&&\:\times\:\Big\{8\big[H_{n+3}(x)-H_{2n+1}(x)\big]
\big[H_{n+3}(x)-H_{2n+1}(x)-\tfrac{2(2xn+4n+1)}{(x+1)(xn+3n+1)}\big]\\
&&\:-\:\,\big[H_n(\tfrac{x}{2})-H_{n+1}(\tfrac{x-1}{2})\big]
\big[H_n(\tfrac{x}{2})-H_{n+1}(\tfrac{x-1}{2})-\tfrac{4(x+n+1)(x^2+3x-2n+1)}{n(x+1)(xn+3n+1)}\big]\\
&&\:+\:\,8H_{n+3}^{\langle2\rangle}(x)-4H_{2n+1}^{\langle2\rangle}(x)+\tfrac{8(x^2+3x+3n^2-n+1)}{n(x+1)(xn+3n+1)}\Big\}.
 \enm
\end{thm}

\begin{proof}
The case $y=2x$ of \eqref{watson-g} can be manipulated as
 \bmn\label{watson-j}
\sum_{k=0}^nk^2\binm{2n-k}{n}\binm{x+k}{k}=\frac{(x+1)(xn+3n+1)}{n-1}\binm{x+2n+1}{n-2}.
 \emn
 Applying the derivative operator $\mathcal{D}_x$ to both sides of
 \eqref{watson-j}, we get
 \bmn\label{harmonic-c}
&&\xxqdn\sum_{k=0}^n\binm{2n-k}{n}\binm{x+k}{k}k^2H_k(x)=\frac{(x+1)(xn+3n+1)}{n-1}
 \nnm\\
&&\xxqdn\:\times\:\:\binm{x+2n+1}{n-2}\Big\{H_{2n+1}(x)-H_{n+2}(x+1)+\tfrac{n}{xn+3n+1}\Big\}.
 \emn
 Applying the derivative operator $\mathcal{D}_x$ to both sides of
 \eqref{harmonic-c}, we gain
 \bnm
&&\sum_{k=0}^n\binm{2n-k}{n}\binm{x+k}{k}k^2\big\{H^2_k(x)-H_{k}^{\langle2\rangle}(x)\big\}=\frac{(x+1)(xn+3n+1)}{n-1}\\
&&\:\times\:\:\binm{x+2n+1}{n-2}\Big\{H_{n+3}^{\langle2\rangle}(x)-H_{2n+1}^{\langle2\rangle}(x)
+\big[H_{n+3}(x)-H_{2n+1}(x)\big]\\
&&\:\times\:\:\big[H_{n+3}(x)-H_{2n+1}(x)-\tfrac{2(2xn+4n+1)}{(x+1)(xn+3n+1)}\big]+\tfrac{2n}{(x+1)(xn+3n+1)}\Big\}.
 \enm
The sum of Theorem \ref{thm-e} and the last equation offers Theorem
\ref{thm-i}.
\end{proof}

Taking $x=p$ with $p\in \mathbb{N}_0$ in Theorem \ref{thm-i} and
exploiting \eqref{watson-j} and \eqref{harmonic-c}, we have the
summation formula on harmonic numbers.

\begin{corl} \label{corl-i}
Let $p$ be a nonnegative integer. Then
 \bnm
&&\xqdn\qqdn\sum_{k=0}^n\binm{2n-k}{n}\binm{p+k}{k}k^2H^2_{p+k}
=\frac{(p+1)(pn+3n+1)}{2(n-1)}\binm{p+2n+1}{n-2}\\
&&\xqdn\qqdn\,\,\times\:\Big\{2(H_{p+n+3}-H_{p+2n+1})\big(H_{p+n+3}-H_{p+2n+1}-2H_{p+1}-\tfrac{2n}{pn+3n+1}\big)+2H^2_p\\
&&\xqdn\qqdn\:+\:\tfrac{4(2pn+4n+1)}{(p+1)(pn+3n+1)}H_p+\tfrac{2(p^2+3p+3n^2-n+1)}{n(p+1)(pn+3n+1)}+2H_{p+n+3}^{\langle2\rangle}
-H_{p+2n+1}^{\langle2\rangle}-H_{p}^{\langle2\rangle}-C_n(p)\Big\},
 \enm
where $C_n(p)$ has been displayed in Corollary \ref{corl-f}.
\end{corl}

When $p=0$, Corollary \ref{corl-i} creates the concise harmonic
 number identity:
\bnm
 &&\xxqdn\sum_{k=0}^{n}\binm{2n-k}{n}k^2H^2_{k}=\frac{3n+1}{2(n-1)}\binm{2n+1}{n-2}
\Big\{2\big[H_{2n+1}-H_{n+3}\big]\big[H_{2n+1}-H_{n+3}+\tfrac{8n+2}{3n+1}\big]
\\&&\xxqdn\:\,-\:
\big[H_{2n+1}-H_{n}\big]\big[H_{2n+1}-H_{n}-\tfrac{2(n+1)(2n-1)}{n(3n+1)}\big]
+2H_{n+3}^{\langle2\rangle}-H_{2n+1}^{\langle2\rangle}+\tfrac{2(3n^2-n+1)}{n(3n+1)}\Big\}.
 \enm

\textbf{Remark}: \eqref{harmonic-a}, \eqref{harmonic-b} and
\eqref{harmonic-c} can also be derived by means of the derivative
operator and Chu-Vandermonde convolution:
 \bnm
\sum_{k=0}^n\binm{x}{k}\binm{y}{n-k}=\binm{x+y}{n}.
 \enm
The reader may refer to Wei et. al. \cito{wei-c} for details. Some
similar results can be seen in Wei et. al. \cito{wei-d}.

 \textbf{Acknowledgments}

 The work is supported by the National Natural Science Foundation of China (No. 11301120).



\end{document}